\documentclass[]{article}
\usepackage[utf8]{inputenc}
\usepackage{amsmath,amssymb,amsthm,mathtools}
\usepackage{graphicx}

\newtheorem{defn}{Definition}
\newtheorem{thm}{Theorem}

\newtheorem{cor}{Corollary}
\newtheorem{lem}{Lemma}
\newtheorem{prop}{Proposition}
\newtheorem*{prob}{Problem}
\newtheorem*{thm*}{Theorem}
\newtheorem*{ex}{Example}
\newtheorem*{conj}{Conjecture}

\title{Recovery of bivariate band limited functions using scattered translates of the Poisson kernel}
\author{Jeff Ledford}
\date{}

\begin{document}

\maketitle

\begin{abstract}
This paper continues the study of interpolation operators on scattered data.  We introduce the Poisson interpolation operator and prove various properties.  The main result concerns functions in the Paley-Wiener space $PW_{B_\beta}$, and shows that one may recover these functions from their samples on a complete interpolating sequence for $[-\delta,\delta]^2$ by using the Poisson interpolation operator, provided that $0<\beta < (3-\sqrt{8})\delta$.
\end{abstract}

\section{Introduction}
This paper continues the study of interpolation on scattered data.  The basic problem is as follows. 
\begin{prob}
Given a set of sampling nodes $X=\{x_j\}$ and data $Y=\{ y_j \}$, find a function $L(x)$, which satisfies $L(x_j)=y_j$.
\end{prob}
An answer depends on the geometry of the sampling nodes, properties of the data, and desired properties of the solution $L$, e.g. continuity, integrability, etc.  We are interested in the case that the sampling nodes $X$ are a so-called complete interpolating sequence for $L^2([-\pi,\pi]^2)$, and the data satisfies $Y\in l^2$.  These hypotheses have been studied before especially in the univariate case, in fact Lyubarskii and Madych showed in \cite{LM} that one may use tempered splines to solve this problem, while Schlumprecht and Sivakumar showed in \cite{SS} that the same is true using scattered translates of the Gaussian.  Along the same lines, the author showed in \cite{me} that one may use scattered translates of `regular interpolators' to solve this problem.  In addition to solving the interpolation problem, all of these schemes had the property that a limiting parameter allowed one to recover functions from the Paley-Wiener space, however the techniques used do not readily extend to the multivariate case.

An extension to the multivariate case was made by Bailey, Schlumprecht, and Sivakumar in \cite{BSS} using scattered translates of the Gaussian.  Our goal is to show that the same is true for the Poisson kernel, which is a regular interpolator.  We prove our results for $n=2$, and employ similar techniques to those found in \cite{BSS}.  One of the advantages of using the Poisson kernel is that one may consider complete interpolating systems on squares, which may be easily described.  The Gaussian does not share this property.  Our main result is the following.
\begin{thm*}
Let $\delta>0$ be given.  Suppose that $\{x_j:j\in\mathbb{N}   \}$ is a complete interpolating sequence for $[-\delta,\delta]^2$ and $f\in PW_{B_\beta}$ for some $0<\beta < (3-\sqrt{8})\delta$, then the Poisson interpolation operator associated to f, denoted $I_\alpha[f](x)$, satisfies $\lim_{\alpha\to\infty}I_\alpha[f](x)=f(x)$ in $L^2(\mathbb{R}^2)$ and uniformly.
\end{thm*}

The rest of this paper is organized as follows.  In the next section, we cover definitions and basic facts and ends with some examples of complete interpolating sequences.  The subsequent section contains the definition and properties of the Poisson interpolation operator.  The fourth section contains the statements and proofs of the recovery results mentioned above.  Finally, an example and some additional comments are collected in the conclusion along with a conjecture pertaining to $\mathbb{R}^n$.


\section{Definitions and Basic Facts}
We adopt the following convention for the Fourier transform.
\begin{defn}
Let $n\in\mathbb{N}$ and suppose that $g\in L^1(\mathbb{R}^2)$, then the \emph{Fourier transform} of $g$ is denoted $\hat{g}$ and given by 
\[
\hat{g}(\xi) = (2\pi)^{-1}\int_{\mathbb{R}^n}g(x)e^{-i\langle \xi, x  \rangle}dx,
\] 
where $\langle u,v \rangle = \sum_{i=1}^{n}u_iv_i $ is the standard inner product on $\mathbb{R}^2$.
\end{defn}
This convention extends to an $L^2(\mathbb{R}^2)$ isomorphism, and we will occasionally use $\mathcal{F}[g](\xi)$ to represent the Fourier transform of $g\in L^2(\mathbb{R}^2)$.
We note the following Fourier transform pair that will be used throughout the remainder of the paper. For $\alpha>0$ we define the \emph{Poisson kernel} $g_\alpha:\mathbb{R}^2\to \mathbb{R}$ by
\[
g_\alpha(x)= (2\pi)^{-1} \alpha \left(\alpha^2+ \langle x,x \rangle  \right)^{-3/2},
\]
the Fourier transform is given by
\[
\hat{g}_{\alpha}(\xi) = e^{-\alpha |\xi|},
\]
where $|u| =\sqrt{ \langle u,u  \rangle }$.  There is a corresponding formula for $\mathbb{R}^n$, but we will not need it here.

We now define the class of functions from which we will sample.
\begin{defn}
Let $S\subset \mathbb{R}^n$ have positive measure.  Then we define the \emph{Paley-Wiener space} denoted $PW_S$, by
\[
PW_S=\{f\in L^2(\mathbb{R}^n): \mathcal{F}[f](\xi) = 0 \text{ a.e. } \xi\notin S    \}.
\]
\end{defn}
A function $f\in PW_S$ is often called \emph{band limited}.

\begin{defn}
Let $\mathcal{H}$ be a separable Hilbert space with inner product $\langle \cdot, \cdot \rangle $.  We call $\{e_j :j\in\mathbb{Z}  \}$ a \emph{Riesz basis} if every $h\in\mathcal{H}$ has a unique representation of the form
\[
\sum_{j\in\mathbb{Z}}a_je_j, \qquad\text{with}\qquad\{a_j:j\in\mathbb{Z}\}\in l^2.
\]
\end{defn}
To each Riesz basis $ \{e_j:j\in\mathbb{Z}\}$, there is an associated \emph{dual Riesz basis} $\{e^*_{j}:j\in\mathbb{Z}\}$ given by the coefficient functionals.  Additionally, we have the so called \emph{Riesz basis inequality}, which provides a constant $B\geq 1$ such that
\begin{equation}\label{RB}
B^{-1}\| \{a_j:j\in\mathbb{Z}\} \|_{l^2} \leq \| \sum_{j\in\mathbb{Z}}a_je_j  \|_{\mathcal{H}} \leq B \| \{a_j:j\in\mathbb{Z}   \}  \|_{l^2}
\end{equation}
for every $\{a_j:j\in\mathbb{Z}    \}\in l^2$.

\begin{defn}
Let $S\subset \mathbb{R}^n$ have positive measure, then a sequence of points $\{ x_j:j\in\mathbb{Z}  \}$ is said to be a \emph{complete interpolating sequence for S} provided that the corresponding sequence of exponentials $\{ e^{i\langle \cdot, x_j \rangle}: j\in\mathbb{Z}   \}$ is a Riesz basis for $L^2(S)$.
\end{defn}
Henceforth we will abbreviate complete interpolating sequence as CIS.  We want our set of sampling nodes to be a CIS for an appropriate set $S$.  Depending on the geometry of $S$ we have different examples.

\begin{ex}
If $\{x_j:j\in\mathbb{Z}\}$ and $\{ y_j :j\in\mathbb{Z}  \}$ are both CISs for $L^2([-\pi,\pi])$, then $\{ (x_j,y_k):j,k\in\mathbb{Z}  \}$ is a CIS for $L^2([-\pi,\pi]^2)$.
\end{ex}
Our next example is due to Favier and Zalik and may be found in \cite{FZ}.  It is a generalization of Kadec's `$1/4$-theorem' and shows that provided we do not stray too far from the integer lattice the resulting sequence will be a Riesz basis.
\begin{ex}
Let $n\in\mathbb{N}$ and suppose that the sequence $\{x_j:j\in\mathbb{Z}^n   \}$ satisfies
$
|j-x_j|\leq L
$ for all $j\in\mathbb{Z}^n$, where $0<L<\pi^{-1}\arccos( (1-9^{1-n})/\sqrt{2})-1/4$, then $\{x_j:j\in\mathbb{Z}^n   \}$ is a CIS for $L^2([-\pi,\pi]^n)$.
\end{ex}
The bound on $L$ is used to bound an auxiliary quantity in the theorem and is more stringent than we need.  We note that when $n=2$ we may take $L\leq 1/20$.

For $n=2$, Lyubarskii and Rashkovskii have constructed CISs for convex symmetric polygons using the zeros of associated entire functions, details are contained in \cite{LR}.  That paper also suggests that their methods may be extended to higher dimensions, but for our purposes, we will need only the examples for $n=2$.

\begin{lem}
Suppose that $M$ is a convex polygon symmetric about the origin and $\{ x_j: j\in\mathbb{Z}   \}$ is a CIS for $L^2(M)$, if $f\in PW_M$, then $\{f(x_j):j\in\mathbb{Z}  \}\in l^2 $.
\end{lem}

\begin{proof}
Let $e_j=e^{-i\langle x_j,\xi \rangle}$, and $e_j^{*}$ be the associated coordinate functional.  We have
\begin{align*}
\mathcal{F}[f](\xi)=\sum_{j\in\mathbb{Z}}\langle \mathcal{F}[f], e_j  \rangle e_j^*(\xi) = 2\pi\sum_{j\in\mathbb{Z}}f(x_j)e_j^*(\xi), 
\end{align*}
since $\{ e^*_j  :j\in\mathbb{Z} \}$ is also a Riesz basis, we have
\begin{align*}
4\pi^2\sum_{j\in\mathbb{Z}}|f(x_j)|^2\leq C^2 \| \mathcal{F}[f]  \|_{L^2(M)} = C^2\| f \|_{L^2(\mathbb{R}^2)}.
\end{align*}
\end{proof}

We remark that our results will be proved for squares, with straightforward modifications leading to the same result for general convex polygons.


\section{Poisson Interpolants}

Our first goal is to show that there is a solution to the interpolation problem.  In what follows we fix $\delta>0$, define $S_\delta=[-\delta,\delta]^2$, and let $\{x_j:j\in\mathbb{N}\}$ be a fixed but otherwise arbitrary CIS for $L^2(S_\delta)$.  We note that $h(\xi)\in L^2(S_\delta)$ enjoys the representation
\begin{equation}
h(\xi)=\sum_{j\in\mathbb{N}}h_j e^{-i\langle \xi, x_j \rangle }, \qquad \xi\in S_\delta.
\end{equation}
If we fix $a\in \mathbb{R}^2$, then
\[
\left\| h \right\|_{L^2(a+S_\delta)}=\left\| \sum_{j\in\mathbb{N}}h_j e^{-i\langle a,x_j \rangle}e^{-i\langle \xi, x_j \rangle }  \right\|_{L^2(S_\delta)}\leq B^2 \|  h\|_{L^2(S_\delta)},
\]
where $B$ is the associated Riesz basis constant defined in \eqref{RB}.  Thus extending $h(\xi)$ to all of $L^2(\mathbb{R}^2)$, we see that
\[
H(\xi)= \sum_{j=1}^{\infty}h_j e^{-i\langle \xi, x_j \rangle}, \qquad \xi\in\mathbb{R}^2
\]
defines an $L^2_{loc}(\mathbb{R}^2)$ function.  For $m=2,3,4,\dots$, we define $A_m:L^2(S_\delta)\to L^2(S_\delta)$ by
\[
A_m[h](\xi)= H(m\xi)\chi_{S_\delta\setminus(1-1/m)S_\delta}(\xi),
\] 
which, along with its adjoint $A_m^*$, satisfies the following bound
\begin{align}
&\left\|  A_m[h] \right\|^2_{L^2(S_\delta)} = \int_{S_\delta\setminus(1-1/m)S_\delta}|H(m\xi)|^2d\xi  \nonumber\\
&= m^{-2}\int_{mS_\delta\setminus(m-1)S_\delta}|H(\xi)|^2 \leq 4(m-1)m^{-2} B^4 \| h  \|^2_{L^2(S_\delta)} \label{A bnd}
\end{align}
The last inequality comes from the fact that $4(m-1)$ copies of $S_\delta$ are needed to cover $m S_\delta\setminus (m-1)S_\delta$.  The following proposition guarantees that we can solve the interpolation problem.

\begin{prop}
Let $\delta>0$ be given and suppose that $\{x_j:j\in\mathbb{N}\} $ is a CIS for $S_\delta$.  If $\{a_j:j\in\mathbb{N}\}\in l^2$, then for all $\alpha>0$ we have
\[
\dfrac{e^{-\sqrt{2}\alpha\delta}}{B^2} \sum_{j=1}^{\infty}|a_j|^2 \leq 2\pi \sum_{j,k=1}^{\infty}a_j\bar{a}_kg_{\alpha}(x_k-x_j) \leq \left( B^2 + \dfrac{4B^6e^{-\alpha\delta}}{(1-e^{-\alpha\delta})^2}  \right) \sum_{j=1}^{\infty}|a_j|^2,
\]
where $B>0$ is the Riesz basis constant defined in \eqref{RB}.
\end{prop}

\begin{proof}  
We begin by proving the upper bound via the Fourier inversion formula and the dominated convergence theorem.  Letting $H_M(\xi)=\sum_{j=1}^{M}a_j e^{-\langle \xi, x_j \rangle}$, we have
\begin{align*}
&2\pi\sum_{j=1}^{M}\sum_{k=1}^{N}a_j\bar{a}_k g\alpha(x_k-x_j) = \int_{\mathbb{R}^2}e^{-\alpha|\xi|} H_M(\xi)\overline{H_N(\xi)}d\xi  \\
=&\int_{S_\delta} e^{-\alpha|\xi|} H_M(\xi)\overline{H_N(\xi)}d\xi \\
& \quad+ \sum_{m=2}^{\infty}m^2  \int_{S_\delta\setminus(1-\frac{1}{m})S_\delta}e^{-\alpha m|\xi|}A_m[H_M](\xi)\overline{A_N[H_N](\xi)}d\xi\\
\leq & \| H_M \|_{L^2(S_\delta)}  \| H_N \|_{l^2(S_\delta)} + \sum_{m=2}^{\infty}m^2e^{-\alpha\delta(m-1)} \| A_m[H_M]\|_{L^2(S_\delta)}\| A_m[H_N]\|_{L^2(S_\delta)}\\
\leq & B^2\sum_{j=1}^{\infty}|a_j|^2 +  4B^4\sum_{m=2}^{\infty}(m-1)e^{-\alpha\delta(m-1)}\| H_M \|_{L^2(S_\delta)}\| H_N \|_{L^2(S_\delta)}\\
\leq & B^2\sum_{j=1}^{\infty}|a_j|^2 + 4B^6 e^{-\alpha\delta}/(1-e^{-\alpha\delta})^2\sum_{j=1}^{\infty}|a_j|^2
\end{align*}
Taking the appropriate limits yields the desired upper bound, we also see that we can interchange the order summation and integration.  This allows us to provide the lower bound
\begin{align*}
& 2\pi \sum_{j,k=1}^{\infty}  a_j\bar{a}_kg_{\alpha}(x_k-x_j) \geq \int_{S_\delta}e^{-\alpha|\xi|}\left| H_{\infty}(\xi)  \right|^2d\xi \\
&\geq e^{-\sqrt{2}\alpha\delta}B^{-2}\sum_{j=1}^{\infty}|a_j|^2.
\end{align*}
\end{proof}
in light of the lemma at the end of the previous section, if $f\in PW_{S_\delta}$, then we may find interpolating coefficients $\{a_j:j\in\mathbb{N}\}\in l^2$ such that
\[
I_\alpha[f](x)=\sum_{j=1}^{\infty}a_jg_\alpha(x-x_j)
\]
satisfies $I_\alpha[f](x_k)=f(x_k)$ for all $k\in\mathbb{N}$.  We call this the \emph{Poisson interpolant} associated to $f$.

\begin{prop}
Let $\alpha,\delta>0$ be given and suppose that $\{x_j:j\in\mathbb{N}\}$ is a CIS for $L^2(S_\delta)$, then $I_\alpha[f](x)\in L^2(\mathbb{R}^2)\cap C(\mathbb{R}^2)$ for any $f\in PW_{S_\delta}$.
\end{prop}
\begin{proof}
This follows from properties of the putative Fourier transform 
\[
\mathcal{F}[I_\alpha[f]](\xi)\text{ `=' } e^{-\alpha|\xi|}\sum_{j=1}^{\infty}a_je^{-i\langle \xi,x_j \rangle}.
\]
Using methods analogous to those used to prove the previous proposition we can show that $\mathcal{F}[I_\alpha[f]](\xi)\in L^1(\mathbb{R}^2)\cap L^2(\mathbb{R}^2)$, which completes the proof.  We will show $\mathcal{F}[I_\alpha[f]](\xi)\in L^1(\mathbb{R}^2)$, the other case being similar.
\begin{align*}
&\int_{\mathbb{R}^2}e^{-\alpha|\xi|}\left| \sum_{j=1}^{\infty}a_je^{-i\langle \xi,x_j \rangle}   \right|d\xi \leq \int_{S_\delta} \left| \sum_{j=1}^{\infty}a_je^{-i\langle \xi,x_j \rangle}   \right|d\xi  \\
&\quad + \sum_{m=2}^{\infty}m^2e^{-\alpha\delta(m-1)}\int_{S_\delta}\left|A_m \left[\sum_{j=1}^{\infty}a_je^{-i\langle \cdot,x_j \rangle} \right](\xi)  \right|d\xi\\
&\leq B^2\sum_{j=1}^{\infty}|a_j|^2+ 4B^6e^{-\alpha\delta}(1-e^{-\alpha\delta})^{-2}\sum_{j=1}^{\infty}|a_j|^2
\end{align*}
We have used the Cauchy-Schwarz inequality to arrive at the final inequality.  Since $\{a_j:j\in\mathbb{N}  \}\in l^2$, $\mathcal{F}[I_\alpha[f]](\xi)\in L^1(\mathbb{R}^2)$.  A similar calculation for $\mathcal{F}[I_\alpha[f]](\xi)\in L^2(\mathbb{R}^2)$ is omitted.
\end{proof}

Our next proposition is crucial to our argument.

\begin{prop}
Let $\alpha,\delta>0$ be given and suppose that $\{x_j:j\in\mathbb{N}  \}$ is a CIS for $S_\delta$.  If $f\in PW_{S_\delta}$, then $\hat{f}$ enjoys the following representation
\begin{equation}\label{represent}
\hat{f}(\xi)=e^{-\alpha|\xi|}u_\alpha(\xi)+\sum_{m=2}^{\infty}m^2 A^{*}_{m}\left[ e^{-\alpha m|\cdot|}A_m[u_\alpha](\cdot)  \right](\xi),
\end{equation}
for almost every $\xi\in S_\delta$.  Here $u_\alpha(\xi)=\sum_{j=1}^{\infty}a_je^{i\langle \xi, x_j \rangle}$, where $\{ a_j:j\in\mathbb{N}\}$ are the interpolating coefficients in the expansion of $I_\alpha[f](x)$.
\end{prop}
\begin{proof}
The interpolation condition $f(x_k)=I_\alpha[f](x_k)$ leads to
\[
\int_{S_\delta}\hat{f}(\xi)e^{i\langle \xi,x_k \rangle}d\xi = \int_{\mathbb{R}^2}\mathcal{F}[I_\alpha[f]](\xi) e^{i\langle \xi,x_k  \rangle }d\xi
\]
for all $k\in\mathbb{N}$.  Writing $\mathcal{F}[I_\alpha[f]](\xi)=e^{-\alpha|\xi|}u_\alpha(\xi)$, we have
\begin{align*}
&\int_{\mathbb{R}^2}\mathcal{F}[I_\alpha[f]](\xi) e^{i\langle \xi,x_k  \rangle }d\xi = \int_{S_\delta}e^{-\alpha|\xi|}u_{\alpha}(\xi)e^{i\langle \xi,x_k  \rangle }d\xi\\
 &\quad + \sum_{m=2}^{\infty} m^2 \int _{S_\delta} e^{-\alpha m |\xi|} A_m[u_\alpha(\cdot)](\xi)A_m[ e^{i\langle \cdot,x_k  \rangle }](\xi)d\xi \\
=& \int_{S_\delta}e^{-\alpha|\xi|}u_{\alpha}(\xi)e^{i\langle \xi,x_k  \rangle }d\xi \\
&\quad + \sum_{m=2}^{\infty} m^2 \int _{S_\delta} A^{*}_{m}\left[ e^{-\alpha m |\cdot|} A_m[u_\alpha(\cdot)\right](\xi) e^{i\langle \xi,x_k  \rangle }   d\xi \\
&=\int_{S_\delta}\left( e^{-\alpha|\xi|}u_{\alpha}(\xi) + \sum_{m=2}^{\infty} m^2  A^{*}_{m}\left[ e^{-\alpha m |\cdot|} A_m[u_\alpha(\cdot)\right](\xi)  \right) e^{i\langle \xi , x_k  \rangle}d\xi
\end{align*}
Here we have tacitly used the dominated convergence theorem to interchange order of summation and integration.  Since $\{ x_j:j\in\mathbb{N} \}$ is a CIS for $L^2(S_\delta)$, the proposition follows.
\end{proof}
To simplify notation somewhat, we introduce $B_\alpha:L^2(S_\delta)\to L^2(S_\delta)$, given by
\[
B_\alpha[h](\xi)=\sum_{m=2}^{\infty}m^2A^{*}_{m}\left[ e^{-\alpha m |\cdot|A_{m}[h](\cdot)}  \right](\xi).
\]
This operator is bounded, in fact
\begin{equation}\label{b bnd}
\| B_\alpha[h] \|_{L^2(S_\delta)} \leq 4B^4e^{-\alpha\delta}(1-e^{-\alpha\delta})^{-2}\|  h \|_{L^2(S_\delta)}.
\end{equation}
We will also make use of the multiplication operator $M_\alpha:L^2(S_\delta)\to L^2(S_\delta)$ defined by
\[
M_\alpha[h](\xi)=e^{\alpha|\xi|}h(\xi),
\]
which satisfies the obvious bound
\[
\| M_\alpha[h] \|_{L^2(S_\delta)}\leq e^{\sqrt{2}\alpha\delta}\| h \|_{L^2(S_\delta)}.
\]
These operators allow us to write \eqref{represent} as
\begin{equation}\label{op rep}
\hat{f}(\xi)=(I+B_\alpha M_\alpha)\left[\mathcal{F}[I_\alpha[f]]\right](\xi),
\end{equation}
for almost every $\xi\in S_\delta$, where $I$ is the usual identity operator.

\begin{prop}
For $\alpha,\delta>0$ and $f\in PW_{S_\delta}$, we have
\begin{equation}\label{u bnd}
\| u_\alpha \|_{L^2(S_\delta)}\leq e^{\sqrt{2}\alpha\delta}\| \hat{f} \|_{L^2(S_\delta)}.
\end{equation}
\end{prop}
\begin{proof}
We need only calculate $\langle \hat{f}(\xi), u_\alpha(\xi)   \rangle$, using \eqref{represent} we have
\begin{align*}
\langle \hat{f}(\xi), u_\alpha(\xi)   \rangle &= \langle e^{-\alpha|\xi|}u_\alpha(\xi),u_\alpha(\xi) \rangle +  \sum_{m=2}^{\infty}m^2\left\langle e^{-\alpha m |\xi|}A_{m}[u_\alpha](\xi), A_{m}[u_\alpha](\xi)    \right\rangle \\
&\geq \langle e^{-\alpha|\xi|}u_\alpha(\xi),u_\alpha(\xi) \rangle \geq e^{-\sqrt{2}\alpha\delta} \| u_\alpha \|^2_{L^2(S_\delta)}.
\end{align*}
On the other hand, the Cauchy-Schwarz inequality yields
\[
|\langle \hat{f}(\xi), u_\alpha(\xi)    \rangle|\leq \| \hat{f} \|_{L^2(S_\delta)}\| u_\alpha \|_{L^2(S_\delta)}.
\]
The proposition follows from combining the two inequalities.
\end{proof}

\begin{prop}
For $\delta>0$ and $f\in PW_{S_\pi}$, there exists $A(\delta)>0$ such that
\[
\| \mathcal{F}\left[ I_\alpha[f]    \right] \|_{L^2(S_\delta)} \leq (1+8B^4e^{\alpha\delta(\sqrt{2}-1)})\| \hat{f}  \|_{L^2(S_\delta)}
\]
for all $\alpha\geq A(\delta).$

\end{prop}
\begin{proof}
It is clear from \eqref{b bnd}, \eqref{op rep}, and \eqref{u bnd} that we need only find $\alpha$ so large that
\[
(1-e^{-\alpha\delta})^{-2}\leq 2.
\]
For this, we may take $\alpha\geq \delta^{-1}\ln(\sqrt{2}/(\sqrt{2}-1))$.
\end{proof}
In order to make the following results less cumbersome, we refer to the constant found above as $A_\delta$, that is $A_\delta=\delta^{-1}\ln(\sqrt{2}/(\sqrt{2}-1))$.
We have the following useful corollary.
\begin{cor}
For $\delta>0$, $\alpha\geq A_\delta$, and $f\in PW_{S_\delta}$, the operator $(I+B_\alpha M_\alpha):L^2(S_\delta)\to L^2(S_\delta)$ is invertible, and 
\[
\|(I+B_\alpha M_\alpha)^{-1}\hat{f}  \|_{L^2(S_\delta)}\leq 9B^4e^{\alpha\delta(\sqrt{2}-1)}\| \hat{f} \|_{L^2(S_\delta)}.
\]
\end{cor}
\begin{prop}
For $\delta>0$, $\alpha\geq A_\delta$, and $f\in PW_{S_\delta}$, we have
\[
\| \mathcal{F}\left[ I_\alpha[f]   \right] \|_{L^2(\mathbb{R}^2\setminus S_\delta)} \leq 4B^2e^{\alpha\delta(\sqrt{2}-1)}\| \hat{f}  \|_{L^2(S_\delta)}.
\]
\end{prop}
\begin{proof}
This is a calculation involving \eqref{A bnd}.
\begin{align*}
\| \mathcal{F}\left[ I_\alpha[f]   \right] \|^2_{L^2(\mathbb{R}^2\setminus S_\delta)} &= \sum_{m=2}^{\infty}\int_{mS_\delta\setminus (m-1)S_\delta}e^{-2\alpha|\xi|}|u_\alpha(\xi)|^2d\xi\\
&= \sum_{m=2}^{\infty}m^2\int_{S_\delta}e^{-2\alpha|\xi|}\left|A_m[u_\alpha](\xi) \right|^2d\xi\\
&\leq 4B^4 \sum_{m=2}^{\infty}(m-1)e^{-\alpha\delta(m-1)}\| u_\alpha \|_{L^2(S_\delta)}^{2}\\
&\leq 16 B^4 e^{2\alpha\delta(\sqrt{2}-1)}\| \hat{f}  \|_{L^2(S_\delta)}^{2}
\end{align*}
The last inequality comes from summing the geometric series and applying \eqref{u bnd}.  This is the desired result once we take square roots.
\end{proof}
We sum up these results in the following theorem.
\begin{thm}
For $\delta>0$, $\alpha\geq A_\delta$, and $f\in PW_{S_\delta}$, the Poisson interpolation operator $I_\alpha: PW_{S_\delta}\to L^2(\mathbb{R}^2)$ is bounded.  We have
\[
\| I_\alpha[f] \|_{L^2(\mathbb{R}^2)}\leq 13 B^4 e^{\alpha\delta(\sqrt{2}-1)}\| f \|_{L^2(\mathbb{R}^2)}.
\]
\end{thm}
\begin{proof}
Since the Fourier transform is an isometry, we need only combine the previous two propositions.
\end{proof}


\section{Recovery Results}

Now that we know the Poisson interpolation operator is bounded, we turn our attention to trying to recover a given function $f\in PW_{S_\delta}$.  Unfortunately, we cannot recover all $f\in PW_{S_\delta}$, however, if $\hat f$ is concentrated enough, recovery is possible.  With this in mind, for $0<\beta < \delta$, we let $B_\beta=\{\xi\in \mathbb{R}^2 : |\xi|\leq \beta  \}$ be the ball of radius $\beta$ in $\mathbb{R}^2$.  Our results will focus on $f\in PW_{B_\beta}$.  Our first theorem deals with the $L^2(\mathbb{R}^2)$ limit.


\begin{thm}
Let $\delta>0$ be given.  If  $0< \beta <(3-\sqrt{8})\delta $ and $f\in PW_{B_\beta}$, then 
\[
\lim_{\alpha\to\infty}\| f- I_\alpha [f]   \|_{L^2(\mathbb{R}^2)} =0
\]
\end{thm}
\begin{proof}
Before beginning, we note that we must take $\alpha\geq A_\delta$, but this poses no problem since $\alpha\to \infty$.  From the Fourier isometry, it is enough to consider $\| \hat{f}-\mathcal{F}\left[ I_\alpha [f]  \right]  \|_{L^2(\mathbb{R}^2)} = \| \hat{f}-\mathcal{F}\left[ I_\alpha [f]  \right]  \|_{L^2(S_\delta)}+ \| \mathcal{F}\left[ I_\alpha [f]  \right]  \|_{L^2(\mathbb{R}^2\setminus S_\delta)}$.  We estimate these norms separately.
We begin with $L^2(S_\delta)$ and use Corollary 1 as well as \eqref{b bnd}:

\begin{align}
\nonumber \| \hat{f}- \mathcal{F}\left[ I_\alpha [f]  \right]  \|_{L^2(S_\delta)} 
 = & \| \left( I-(I+B^\alpha M_\alpha)^{-1} \right) [\hat{f}]\|_{L^2(S_\delta)}\\
\nonumber = & \| \left( I+ B_\alpha M_\alpha  \right)^{-1}B_\alpha M_\alpha [\hat{f}]   \|_{L^2(S_\delta)}\\
\nonumber  \leq& 9B^4e^{\alpha\delta(\sqrt{2}-1)}\| B_\alpha M_\alpha [\hat{f}]  \|_{L^2(S_\delta)} \\
\nonumber \leq & 72 B^8 e^{\alpha\delta(\sqrt{2}-2)}\| M_\alpha[\hat f] \|_{L^2(S_\delta)}\\
  \leq & 72 B^8 e^{\alpha(\beta + \delta(\sqrt{2}-2))}\|\hat f  \|_{L^2(S_\delta)} \label{BND1}.
\end{align}
This is the desired estimate for $L^2(S_\delta)$.  Notice that the limit tends to 0 since $2-\sqrt{2}< 3-\sqrt{8}$.  For $L^2(\mathbb{R}^2\setminus S_\delta)$, we have
\begin{align}
\nonumber&\| \mathcal{F}\left[ I_\alpha [f]  \right]  \|^{2}_{L^2(\mathbb{R}^2\setminus S_\delta)} \leq \sum_{m=2}^{\infty}m^2e^{-2\alpha\delta(m-1)}\| A_m[u_\alpha]  \|^{2}_{L^2(S_\delta)} \\
\nonumber&\leq 8 B^4 e^{-2\alpha\delta}\| u_\alpha \|^2_{L^2(S_\delta)}  \\
\nonumber&\leq 8 B^4 e^{-2\alpha\delta}\| M_\alpha\left[ \mathcal{F}[I_\alpha[f]]   \right] \|^2_{L^2(S_\delta)}\\
\nonumber& \leq 8 B^4 e^{-2\alpha\delta} \left( e^{\sqrt{2}\alpha\delta}\| \mathcal{F}[I_\alpha[f]] - \hat f  \|_{L^2(S_\delta)} +  e^{\alpha\beta}\| \hat f \|_{L^2(S_\delta)}   \right)^2\\
&\leq 8 B^4 e^{-2\alpha\delta} \left( 73 B^8 e^{\alpha(\beta + \delta(\sqrt{8}-2))} \right)^2 \| \hat f  \|_{L^2(S_\delta)}^{2}.\label{BND2}
\end{align}
This means that
\[
\| \mathcal{F}\left[ I_\alpha [f]  \right]  \|_{L^2(\mathbb{R}^2\setminus S_\delta)} \leq 219 B^{10} e^{\alpha(\beta + \delta(\sqrt{8}-3))} \| \hat f  \|_{L^2(S_\delta)},
\]
so the limit tends to 0.  Since both portions tend to 0, the theorem is proved.
\end{proof}

We may now prove the following result about uniform convergence.

\begin{thm}
Let $\delta>0$ be given.  If  $0< \beta <(3-\sqrt{8})\delta $ and $f\in PW_{B_\beta}$, then 
\[
\lim_{\alpha\to\infty}\left| f(x)- I_\alpha [f] (x) \right| =0
\]
uniformly on $\mathbb{R}^2$.
\end{thm}
\begin{proof}
We use the Fourier inversion theorem and the Cauchy-Schwarz inequality.
\begin{align*}
&2\pi \left| f(x)- I_\alpha [f] (x) \right| = \left| \int_{\mathbb{R}^2}\hat{f}(\xi) e^{i\langle \xi,x  \rangle }d\xi   - \int_{\mathbb{R}^2}\mathcal{F}[I_\alpha[f]](\xi) e^{i\langle \xi,x  \rangle }d\xi \right|\\
\leq &  \int_{S_\delta} \left| \hat{f}(\xi)- \mathcal{F}[I_\alpha[f]](\xi)   \right|d\xi + \sum_{m=2}^{\infty}m^2 \int_{S_\delta} \left|A^*_{m}\left[e^{-\alpha m |\cdot|}A_m[u_\alpha](\cdot)  \right](\xi)  \right|d\xi \\
\leq & 2\delta \| \hat{f}- \mathcal{F}[I_\alpha[f]] \|_{L^2(S_\delta)} + \delta 8B^4e^{-\alpha\delta}\| u_\alpha \|_{L^2(S_\delta)}\\
\leq & 2\delta\left(   72 B^8 e^{\alpha(\beta +\delta(\sqrt{2}-2)   )}    \|\hat f\|_{L^2(S_\delta)} \right.\\
& \quad + 576 B^{12}e^{\alpha( \beta+ \delta(\sqrt{8}-3)   )}   \|\hat f\|_{L^2(S_\delta)}  
 +       8B^4e^{\alpha(\beta-\delta)}    \|\hat f\|_{L^2(S_\delta)}    \left.\right)
\end{align*}
Here we have used \eqref{BND1} and an argument similar to the one used to find \eqref{BND2} to arrive at the last inequality.  These three terms all tend to 0 provided that $0<\beta < (3-\sqrt{8})\delta$.  Since none of the estimates depend on $x\in\mathbb{R}^2$, the convergence is uniform.
\end{proof}


\section{Conclusions}

The argument here may be amended to include zonoids in higher dimensions as is done in \cite{BSS} for the Gaussian interpolation operator.  In general one obtains a similar result for zonoids with fewer faces than required by the Gaussian.  This is advantageous, but especially so in $\mathbb{R}^2$ where we can use squares.  Since Kadec's theorem has been extended to cubes in higher dimensions, we have an easy criterion which allows us to find a CIS.  We present the following example.
\begin{ex}
Suppose that $\{x_j: j\in\mathbb{Z}  \}$ satisfies $|x_j-j|<1/20$ for all $j\in\mathbb{Z}^2$, and that $f\in PW_{B_{\pi/6}}$, then the Poisson interpolation operator $I_\alpha [f](x)$ converges to $f(x)$ in $L^2(\mathbb{R}^2)$ and uniformly on $\mathbb{R}^2$.   
\end{ex}
Unfortunately, if $n>2$, CISs for cubes cannot be used in the procedure for the Poisson interpolation operator.  What seems necessary to be able to use a CIS for a cube is slower decay.  The trade-off is that it's difficult to find the right speed and the kernel $g(x)$ becomes complicated, nevertheless, we make the following conjecture.

\begin{conj}
Let $n\in \mathbb{N}$, there exists $\omega(n)>0$ such that the interpolation operator $I_\alpha$ associated to 
\[
g_\alpha(x)=\int_{\mathbb{R}^n}e^{-\alpha|\xi|^{\omega}}e^{i\langle \xi, x  \rangle}d\xi
\]
admits the use of CISs for cubes in $\mathbb{R}^n$ to recover functions as in Theorems 2 and 3.
\end{conj} 

The author has verified this for $n=3$.  Using the same argument we have the following theorem.
\begin{thm*}
For $ n=3$, let $\delta>0$ be given and $\{ x_j: j\in\mathbb{N} \}$ a CIS for $[-\delta,\delta]^3$.  Then if $0<\omega< 2(1-\ln(2)/\ln(3))$, the conjecture above is true and we can recover $f\in PW_{B_\beta}$ where $0< \beta < (3-2(3)^{\omega/2})^{1/\omega}\delta$.
\end{thm*}



\begin{thebibliography}{00}

\bibitem{BSS}
B.A. Bailey, T. Schlumprecht, \and N. Sivakumar,
``Nonuniform sampling and recovery of multidimensional bandlimited functions by Gaussian radial-basis functions,''
J. Fourier Anal. Appl. {\bf 17} (2011), no. 3, 519-533.

\bibitem{FZ}
S.J. Favier, \and R.A. Zalik,
``On the stability of frames and Riesz bases,''
Appl. Comput. Harmon. Anal. {\bf 2} (1995), no. 2, 160-173.


\bibitem{me}
J. Ledford,
``Recovery of Paley-Wiener functions using scattered translates of regular interpolators,"
J. Approx. Theory {\bf 173} (2013), 1-13.


\bibitem{LM}
Y. Lyubarskii, \and W. Madych, 
``Recovery of Irregularly Sampled Band Limited Functions via Tempered Splines,'' 
J.  Funct. Anal. \text{\bf{125}} (1994), 201-222.

\bibitem{LR}
Y. Lyubarskii, \and A. Rashkovskii,
``Complete interpolating sequences for Fourier transforms supported by convex symmetric polygons,''
Ark. Mat. {\bf 38} (2000), no. 1, 139-170.

\bibitem{SS}
T. Schlumprecht, \and N. Sivakumar,
``On the Sampling and Recovery of Bandlimited Functions via Scattered Translates of the Gaussian,''
J. Approx. Theory \text{\bf{159}} (2009), 128-153.

\end{thebibliography}
\end{document}